\documentclass[12pt]{amsart}
\usepackage{amsmath,amssymb}
\usepackage{amsfonts}
\usepackage{amsthm}
\usepackage{latexsym}
\usepackage{graphicx}

\def\supp{\mbox{supp}}

\def\Ric{\mbox{Ric}}

\def\p{\partial}

\def\R{\mathbb{R}}

\def\vv<#1>{\langle#1\rangle}

\def\XXint#1#2{\setbox0=\hbox{$#1{#2}{\int}$}{#2}\kern-.5\wd0 }

\def\XXint#1#2#3{{\setbox0=\hbox{$#1{#2#3}{\int}$}
     \vcenter{\hbox{$#2#3$}}\kern-.5\wd0}}

\def\vv<#1>{\langle#1\rangle}
\newtheorem{thm}{Theorem}[section]

\newtheorem{lem}{Lemma}[section]
\newtheorem{prop}{Proposition}[section]
\newtheorem{cor}{Corollary}[section]
\theoremstyle{definition}

\theoremstyle{remark}

\newtheorem{rem}{Remark}[section]

\numberwithin{equation}{section}

\begin{document}

\title{Li-Yau multiplier set and optimal Li-Yau gradient estimate on hyperbolic spaces}

\author{Chengjie Yu$^1$}
\address{Department of Mathematics, Shantou University, Shantou, Guangdong, 515063, China}
\email{cjyu@stu.edu.cn}
\author{Feifei Zhao}
\address{Department of Mathematics, Shantou University, Shantou, Guangdong, 515063, China}
\email{14ffzhao@stu.edu.cn}
\thanks{$^1$Research partially supported by an NSF project of China with contract no. 11571215.}

\renewcommand{\subjclassname}{%
  \textup{2010} Mathematics Subject Classification}
\subjclass[2010]{Primary 35K05; Secondary 53C44}
\date{}
\keywords{Heat equation, Li-Yau type gradient estimate, heat kernel}
\begin{abstract}
In this paper, motivated by finding sharp Li-Yau type gradient estimate for positive solution of heat equations on complete Riemannian manifolds with negative Ricci curvature lower bound, we first introduce the notion of Li-Yau multiplier set and show that it can be computed by heat kernel of the manifold. Then, an optimal Li-Yau type gradient estimate is obtained on hyperbolic spaces by using recurrence relations of heat kernels on hyperbolic spaces. Finally, as an application, we obtain sharp Harnack inequalities on hyperbolic spaces.
\end{abstract}
\maketitle\markboth{Yu \& Zhao}{Li-Yau type gradient estimate}
\section{Introduction}
The Li-Yau \cite{LY} gradient estimate:
\begin{equation}\label{eq-LY}
\|\nabla \log u\|^2-\alpha(\log u)_t\leq \frac{n\alpha^2}{2t}+\frac{n\alpha^2k}{2(\alpha-1)}
\end{equation}
for positive solution $u$ of the heat equation on complete Riemannian manifolds with $\Ric\geq -k$ and $k$ a nonnegative constant is of fundamental importance in geometric analysis. Here $\alpha$ is any constant greater than $1$.

On complete Riemannian manifolds with nonnegative Ricci curvature,  by letting $\alpha\to1^+$ in \eqref{eq-LY}, one has
\begin{equation}\label{eq-LY-0}
\|\nabla \log u\|^2-(\log u)_t\leq \frac{n}{2t}.
\end{equation}
This estimate is sharp where the equality can be achieved by the fundamental solution of  $\R^n$. However, \eqref{eq-LY} is not sharp when $k>0$. Finding sharp Li-Yau type gradient estimate for $k>0$ is still an unsolved problem. This is the motivation of this paper. We will assume that $k>0$ without further indications in the rest of this paper.

Li-Yau type gradient estimates are important since they give Harnack inequalities immediately by taking integration on geodesics. Many authors have obtained Li-Yau type gradient estimates in various forms or various settings. For example, in \cite{Ha}, Hamilton obtained a Li-Yau type gradient estimate in matrix form, and in \cite{Ha2}, Hamilton obtained a Li-Yau type gradient estimate in matrix form for Ricci flow. Hamilton's works were extended to the K\"ahler category by  Cao-Ni \cite{CN} and Cao \cite{Cao}, and further extended to $(p,p)$-forms on K\"ahler manifolds by Ni and Niu \cite{NN}. Recently, in \cite{ZZh}, the authors extended the Li-Yau type gradient estimate to metric measure spaces, and in \cite{ZZ0,ZZ,Ca,Ro}, the authors obtained Li-Yau type gradient estimates under integral curvature assumptions. Some other Li-Yau type gradient estimates can  be found in \cite{BQ,BB,BL,CFL,CH,Le,Pe,RY,Ya94,Ya95}. Here, we only mention some of them that are more related to the topic of this paper and compare them.

A slight improvement of \eqref{eq-LY}:
\begin{equation}\label{eq-LYD}
\|\nabla \log u\|^2-\alpha (\log u)_t\leq \frac{n\alpha^2}{2t}+\frac{n\alpha^2k}{4(\alpha-1)}
\end{equation}
was given by Davies \cite{Da}. We will call this the Li-Yau-Davies estimate in the rest of this paper.

In \cite{Ha}, Hamilton obtained
\begin{equation}\label{eq-Ha}
\|\nabla \log u\|^2-e^{2kt}(\log u)_t\leq e^{4kt}\frac{n}{2t}.
\end{equation}
This estimate is sharp in leading term as $t\to 0^+$ comparing to the Li-Yau-Davies estimate \eqref{eq-LYD}.

In \cite{BQ}, Bakry and Qian obtained
\begin{equation}\label{eq-BQ}
\|\nabla \log u\|^2-\left(1+\frac{2}{3}kt\right)(\log u)_t\leq\frac{n}{2t}+\frac{nk}{2}\left(1+\frac{1}{3}kt\right).
\end{equation}
This estimate is also sharp in leading term as $t\to 0^+$. This estimate was also obtained by Li and Xu \cite{LX} by a different method.

In \cite{LX},  Li and Xu obtained
\begin{equation}\label{eq-LX}
\|\nabla \log u\|^2-\left(1+\frac{\sinh(kt)\cosh(kt)-kt}{\sinh^{2}(kt)}\right)(\log u)_t\leq\frac{nk}{2}[\coth(kt)+1].
\end{equation}
It is not hard to see that \eqref{eq-LX} is also sharp in leading term as $t\to 0^+$. Moreover, as pointed out in \cite{YZ1},
 the asymptotic behavior of \eqref{eq-LX} as $t\to\infty$ is the same as \eqref{eq-LYD} with $\alpha=2$. The estimates \eqref{eq-BQ} and \eqref{eq-LX} are extended to a general form in \cite{Qi}. Li-Xu's estimate \eqref{eq-LX} was also obtained by Bakry et al. in \cite{BB} by a different method.

One should note that, in the Li-Yau-Davies estimate \eqref{eq-LYD}, one can choose different $\alpha>1$ for different given $t$. So, the time-dependence of $\alpha$ in the estimates \eqref{eq-Ha}, \eqref{eq-BQ} and \eqref{eq-LX} is not essentially a new feature in Li-Yau type gradient estimates for positive solutions of heat equations. A key feature is that the estimates \eqref{eq-Ha}, \eqref{eq-BQ} and \eqref{eq-LX} are all sharp in leading term as $t\to 0^+$.

For purpose of comparison, we rewrite a Li-Yau type gradient estimate in the following form:
\begin{equation}\label{eq-LY-1}
\beta\|\nabla \log u\|^2-(\log u)_t\leq \gamma.
\end{equation}
For example, for the Li-Yau-Davies estimate \eqref{eq-LYD},
\begin{equation}\label{eq-g-LYD}
\gamma_{LYD}(\beta,t)=\frac{n}{2t}\left(\frac 1\beta +\frac{kt}{2(1-\beta)}\right)
\end{equation}
and $\beta\in (0,1)$. For Hamilton's estimate \eqref{eq-Ha},
\begin{equation}
\beta=\beta_{H}:=e^{-2kt}
\end{equation}
and
\begin{equation}
\gamma=\gamma_{H}:=\frac{n}{2t}e^{2kt}.
\end{equation}
For Bakry-Qian's estimate \eqref{eq-BQ},
\begin{equation}
\beta=\beta_{BQ}:=\frac{1}{1+\frac{2}{3}kt}
\end{equation}
and
\begin{equation}
\gamma=\gamma_{BQ}:=\frac{n}{2t}\cdot\frac{1+kt+\frac13(kt)^2}{1+\frac23kt}.
\end{equation}
For Li-Xu's estimate \eqref{eq-LX},
\begin{equation}
\beta=\beta_{LX}:=\frac{1}{1+\frac{\sinh(kt)\cosh(kt)-kt}{\sinh^{2}(kt)}}
\end{equation}
and
\begin{equation}
\gamma=\gamma_{LX}:=\frac{n}{2t}\cdot\frac{kt[\coth(kt)+1]}{1+\frac{\sinh(kt)\cosh(kt)-kt}{\sinh^{2}(kt)}}.
\end{equation}

For a fixed time $t$, a Li-Yau type gradient estimate
\begin{equation}
\beta_1\|\nabla \log u\|^2-(\log u)_t\leq \gamma_1
\end{equation}
is better than
\begin{equation}
\beta_2\|\nabla \log u\|^2-(\log u)_t\leq \gamma_2
\end{equation}
if $\beta_1\geq \beta_2$ and $\gamma_1\leq \gamma_2$. For example, for a fixed time $t>0$,  $\gamma_{LYD}(\beta,t)$ achieves its minimum
\begin{equation}
\gamma_m(t)=\frac{n}{2t}\left(1+\sqrt\frac{kt}2\right)^2
\end{equation}
at
\begin{equation}
\beta_m(t)=\frac{1}{1+\sqrt\frac{kt}2}.
\end{equation}
Therefore, for each time $t>0$, the Li-Yau-Davies estimate \eqref{eq-LYD} for $\beta=\beta_m(t)$ is better than the Li-Yau-Davies estimate for $\beta<\beta_m(t)$.

In the Li-Yau-Davies estimate, let $\beta=\beta_{H}:=e^{-2kt}$. Then,
\begin{equation}
\gamma_{LYD}(\beta_H(t),t)=\frac{n}{2t}e^{2kt}+\frac{nk}{4(1-e^{-2kt})}>\gamma_{H}(t).
\end{equation}
Comparing this to Hamilton's estimate \eqref{eq-Ha}, it seems that Hamilton's estimate is better than the Li-Yau-Davies estimate \eqref{eq-LYD} for $\beta=\beta_H(t)$ for all time.
However, note that $\beta_H(t)<\beta_m(t)$ and
\begin{equation}
\gamma_m(t)<\gamma_{H}(t)
\end{equation}
when $t$ is large enough. So, when $t$ is large enough, the Li-Yau-Davies estimate \eqref{eq-LYD} is better than \eqref{eq-Ha}. More precisely, let $t_H>0$ be the intersection point of $\gamma_m(t)$ and $\gamma_H(t)$. Then, when $t\geq t_H$, since $\beta_m(t)\geq \beta_H(t)$ and $\gamma_{H}(t)\geq \gamma_m(t)$, Li-Yau-Davies estimate \eqref{eq-LYD} is better than Hamilton's estimate \eqref{eq-Ha}, and when $t< t_H$, since $\gamma_m(t)>\gamma_H(t)$, Hamilton's estimate \eqref{eq-Ha} is better than the Li-Yau-Davies estimate \eqref{eq-LYD} with $\beta\leq \beta_H(t)$. However, for $t<t_H$, Hamilton's estimate \eqref{eq-Ha} is not better than the Li-Yau-Davies estimate \eqref{eq-LYD} with $\beta\in (\beta_H(t),1)$ although it is sharp in leading term as $t\to 0^+$.

For Bakry-Qian's estimate \eqref{eq-BQ}, since $\gamma_{BQ}(t)<\gamma_m(t)$ for any $t>0$, Bakry-Qian's estimate \eqref{eq-BQ} is better than the Li-Yau-Davies estimate \eqref{eq-LYD} with $\beta\leq \beta_{BQ}(t)$ for any $t>0$. Moreover, since $\beta_{BQ}(t)>\beta_H(t)$ and $\gamma_{BQ}(t)<\gamma_{H}(t)$ for any $t>0$, Bakry-Qian's estimate \eqref{eq-BQ} is better than Hamilton's estimate \eqref{eq-Ha}.

For Li-Xu's estimate \eqref{eq-LX}, let $t_{LX}$ be the intersection point of $\gamma_{LX}(t)$ and $\gamma_m(t)$. Then, when $t\leq t_{LX}$, Li-Xu's estimate is better than the Li-Yau-Davies estimate for $\beta\leq \beta_{LX}(t)$ since $\gamma_m(t)\geq\gamma_{LX}(t)$ for $t\leq t_{LX}$. When $t> t_{LX}$, since $\beta_m(t)<\beta_{LX}(t)$ (see Proposition \ref{prop-LXLY} in the Appendix) and $$\gamma_{LYD}(\beta_{LX}(t),t)>\gamma_{LX}(t)>\gamma_m(t).$$
 Let $\beta_-(t)<\beta_m(t)<\beta_+(t)$ be such that
\begin{equation}
\gamma_{LYD}(\beta_+(t),t)=\gamma_{LYD}(\beta_-(t),t)=\gamma_{LX}(t).
\end{equation}
Then, Li-Xu's estimate \eqref{eq-LX} is better than the Li-Yau-Davies estimate \eqref{eq-LYD} for $\beta\in (0,\beta_-(t)]\cup[\beta_+(t),\beta_{LX}(t)]$ at each time $t>t_{LX}$.

For the comparison of Li-Xu's estimate \eqref{eq-LX} and Bakry-Qian's estimate \eqref{eq-BQ}, although one has $\beta_{BQ}(t)<\beta_{LX}(t)$ for any $t>0$, one can not conclude that Li-Xu's estimate \eqref{eq-LX} is better than Bakry-Qian's estimate \eqref{eq-BQ} when $t$ is large enough, since
\begin{equation}
\lim_{t\to\infty}\gamma_{BQ}(t)=\frac{nk}{4}<\frac{nk}{2}=\lim_{t\to \infty}\gamma_{LX}(t).
\end{equation}
Moreover, since $\beta_H(t)<\beta_{LX}(t)$ and $\gamma_{H}(t)>\gamma_{LX}(t)$ for any $t>0$ (see Proposition \ref{prop-LXHa} in the Appendix), Li-Xu's estimate \eqref{eq-LX} is better than Hamilton's estimate \eqref{eq-Ha}.

In summary, Hamilton's estimate \eqref{eq-Ha}, Bakry-Qian's estimate \eqref{eq-BQ} and Li-Xu's estimate \eqref{eq-LX} are better than the Li-Yau-Davies estimate \eqref{eq-LYD} only for certain range of $t$ and certain range of $\beta$. Bakry-Qian's \eqref{eq-BQ} and Li-Xu's \eqref{eq-LX} estimates are both better than Hamilton's estimate \eqref{eq-Ha}.

Motivated by the comparisons above, we introduce the following notion of Li-Yau multiplier set. Let $(M^n,g)$ be a complete Riemannian manifold and
\begin{equation}
\mathcal P(M,g)=\{u\in C^\infty(\R^+\times M)\ |\ u>0\ \mbox{and}\ u_t-\Delta_g u=0\}.
\end{equation}
For $x\in M$ and $u\in \mathcal P(M,g)$, define the Li-Yau multiplier set of $u$ at $x$ as
\begin{equation}
S(u,x)=\{(t,\beta,\gamma)\in \R^+\times[0,+\infty)\times\R\ |\  [\beta\|\nabla \log u\|^2-(\log u)_t](t,x)\leq \gamma\},
\end{equation}
and the Li-Yau multiplier set of $u$ at $x$ and time $t$ as
\begin{equation}
S_t(u,x)=\{(\beta,\gamma)\in [0,+\infty)\times\R\ |(t,\beta,\gamma)\in S(u,x) \}.
\end{equation}
Moreover, we define the Li-Yau multiplier set of $(M,g)$ at $x$ as
\begin{equation}
S(M,g,x)=\bigcap_{u\in \mathcal P(M,g)}S(u,x),
\end{equation}
the Li-Yau multiplier set of $(M,g)$ at $x$  and time $t$ as
\begin{equation}
S_t(M,g,x)=\bigcap_{u\in \mathcal P(M,g)}S_t(u,x),
\end{equation}
the Li-Yau multiplier set of $(M,g)$ as
\begin{equation}
S(M,g)=\bigcap_{x\in M}S(M,g,x),
\end{equation}
and the Li-Yau multiplier set of $(M,g)$ at time $t$ as
\begin{equation}
S_t(M,g)=\bigcap_{x\in M}S_t(M,g,x).
\end{equation}
A key observation of the Li-Yau multiplier set is as follows:
\begin{thm}\label{thm-LY-set}
Let $(M^n,g)$ be a complete Riemannian manifold with Ricci curvature bounded from below and $H(t,x,y)$ be its heat kernel. Then,
\begin{equation}
S(M,g,x)=\bigcap_{y\in M}S(H(\cdot,\cdot,y),x).
\end{equation}
In other words, $(t,\beta,\gamma)\in S(M,g,x)$ if and only if
\begin{equation}
\left[\beta\|\nabla_x \log H \|^2-(\log H)_t\right](t,x,y)\leq \gamma
\end{equation}
for any $y\in M$. As a consequence,
\begin{equation}
S(M,g)=\bigcap_{x,y\in M}S(H(\cdot,\cdot,y),x).
\end{equation}
In other words, $(t,\beta,\gamma)\in S(M,g)$ if and only if
\begin{equation}
\left[\beta\|\nabla_x \log H \|^2-(\log H)_t\right](t,x,y)\leq \gamma
\end{equation}
for any $x,y\in M$.
\end{thm}
\noindent This observation tells us that to check if a triple $(t,\beta,\gamma)$ belongs to $S(M,g)$, one only need to check if
\begin{equation}
\beta\|\nabla \log u\|^2-(\log u)_t\leq \gamma
\end{equation}
is true for the heat kernel. This highly simplifies the computation of $S(M,g)$.

By using the expression
\begin{equation}
H(t,x,y)=(4\pi t)^{-\frac n2}e^{-\frac{\|x-y\|^2}{4t}}
\end{equation}
of heat kernel on the Euclidean space $\R^n$ with standard metric $g_{E}$, it is not hard to see that
\begin{equation}\label{eq-LY-E}
S(\R^n,g_{E})=\left\{(t,\beta,\gamma)\ |\ t>0,\beta \in [0,1],\mbox{ and } \gamma\geq \frac{n}{2t}\right\}.
\end{equation}
The Li-Yau estimate \eqref{eq-LY-0} is equivalent to that
\begin{equation}
S(M^n,g)\supset S(\R^n,g_{E})
\end{equation}
for any complete Riemannian manifold $(M^n,g)$ with nonnegative Ricci curvature. Motivated by this, one may reformulate the problem of finding sharp Li-Yau type gradient estimate as follows: Let $M_{\kappa}^n$ be the space form of dimension $n$ with constant sectional curvature $\kappa$ and with standard metric $g_{\kappa}$. Do we have
\begin{equation}
S(M^n,g)\supset S(M_\kappa^n,g_\kappa)
\end{equation}
for any complete Riemannian manifold $(M^n,g)$ with Ricci curvature not less than $(n-1)\kappa$? Laying aside the problem, one still has another problem of finding $S(M_\kappa^n,g_\kappa)$ for $\kappa\neq 0$. Motivated by the work \cite{DM} of Davies and Mandouvalos,  we will discuss this problem with $\kappa=-1$ in this paper.

More precisely, let $K_n(t,r(x,y))$ be the heat kernel of the $n$-dimensional hyperbolic space, then by using the recurrence relation:
\begin{equation}\label{eq-re-s2}
K_{n+2}=-\frac{e^{-nt}}{2\pi\sinh r}\p_r K_n,
\end{equation}
we are able to show the following Li-Yau type gradient estimate on hyperbolic space of odd dimension.
\begin{thm}\label{thm-odd} Let $(M^n,g)$ be an odd dimensional complete Riemannian manifold with constant sectional curvature $-1$. Then $$\left(t, \beta,\frac{n}{2t}+\frac{(n-1)^2}{4(1-\beta)}\right)\in S(M,g),$$
for any $t>0$ and $\beta\in [0,1)$. In other words,
\begin{equation}
\beta\|\nabla \log u\|^2-(\log u)_t\leq \frac{n}{2t}+\frac{(n-1)^2}{4(1-\beta)}
\end{equation}
for any $\beta\in [0,1)$ and any positive solution $u$ of the heat equation on $(M,g)$.
\end{thm}
\noindent This estimate is sharp in leading term as $t\to 0^+$ and sharp as $t\to \infty$ for hyperbolic spaces.

For hyperbolic space of even dimension, by using the recurrence relation:
\begin{equation}\label{eq-re-s1}
K_n(t,r)=\sqrt 2e^{\frac{(2n-1)t}{4}}\int_r^\infty \frac{K_{n+1}(t,\rho)\sinh\rho}{(\cosh\rho-\cosh r)^\frac12}d\rho.
\end{equation}
 we are only able to obtain a weaker conclusion:
\begin{thm}\label{thm-even} Let $(M^n,g)$ be an even dimensional complete Riemannian manifold with constant sectional curvature $-1$. Then $$\left(t, \beta,\frac{n+1}{2t}+\frac{(n-1)^2}{4(1-\beta)}\right)\in S(M,g),$$
for any $t>0$ and $\beta\in [0,1)$.
In other words,
\begin{equation}
\beta\|\nabla \log u\|^2-(\log u)_t\leq \frac{n+1}{2t}+\frac{(n-1)^2}{4(1-\beta)}
\end{equation}
for any $\beta\in [0,1)$ and any positive solution $u$ of the heat equation on $(M,g)$.
\end{thm}
\noindent This estimate is better than the Li-Yau-Davies estimate for large $t$, although it is not sharp in leading term as $t\to 0^+$. In fact, for the hyperbolic plane $\mathbb H^2$ (see Proposition \ref{prop-dim-2} in the Appendix),
\begin{equation}\label{eq-dim-2}
-(\log K_2)_t(t,0)>\frac{1}{t}+\frac{1}{4}.
\end{equation}
So, we can not expect the same conclusion as in Theorem \ref{thm-odd} holds on the hyperbolic plane.

It is not hard to see that $S_{t_1}(M,g)\subset S_{t_2}(M,g)$ for any $0<t_1<t_2$ (see (6) of Proposition \ref{prop-ele-p-S}). So, we define the Li-Yau multiplier set at time infinity as
\begin{equation}
S_\infty(M,g)=\cup_{t>0}S_t(M,g).
\end{equation}
By \eqref{eq-LY-E}, it is clear that
\begin{equation}
S_\infty(\R^n,g_E)=[0,1]\times \R_+.
\end{equation}
A direct corollary of Theorem \ref{thm-odd} and Theorem \ref{thm-even} is as follows.
\begin{cor}
Let $(M^n,g)$ be a complete Riemannian manifold with constant sectional curvature $-1$. Then, $$S_\infty(M,g)\supset\left\{(\beta,\gamma)\ \Bigg|\beta\in [0,1)\ {\rm and}\ \gamma>\frac{(n-1)^2}{4(1-\beta)}\right\}.$$
\end{cor}

Finally, by a standard argument as in \cite{LY}, we have the following sharp  Harnack inequality.
\begin{thm}\label{thm-harnack}
\begin{enumerate}
\item Let $(M^n,g)$ be an odd dimensional complete Riemannian manifold with constant sectional curvature $-1$ and  $u$ be a positive solution to the heat equation on $M$. Then, for any $x_1,x_2\in M$ and $0<t_1<t_2$,
    \begin{equation*}
    u(x_1,t_1)\leq \left(\frac{t_2}{t_1}\right)^\frac n2 \exp\left(\frac{r^2(x_1,x_2)}{4(t_2-t_1)}+\frac{(n-1)^2}4(t_2-t_1)+\frac{n-1}{2}r(x_1,x_2)\right)u(x_2,t_2).
    \end{equation*}
\item Let $(M^n,g)$ be an even dimensional complete Riemannian manifold with constant sectional curvature $-1$ and $u$ be a positive solution to the heat equation on $M$. Then, for any $x_1,x_2\in M$ and $0<t_1<t_2$,
    \begin{equation*}
    u(x_1,t_1)\leq \left(\frac{t_2}{t_1}\right)^\frac {n+1}2 \exp\left(\frac{r^2(x_1,x_2)}{4(t_2-t_1)}+\frac{(n-1)^2}4(t_2-t_1)+\frac{n-1}{2}r(x_1,x_2)\right)u(x_2,t_2).
    \end{equation*}
\end{enumerate}
Here $r(x_1,x_2)$ means the distance between $x_1$ and $x_2$.
\end{thm}

The recurrence relations \eqref{eq-re-s2} and \eqref{eq-re-s1} were proved in \cite{DM} by using Selberg's transform. For a simple direct proof and some similar recurrence relations on spheres, see \cite{YZ2}.

The rest of this paper is organized as follows. In Section 2, we introduce some elementary properties of Li-Yau multiplier sets and prove Theorem \ref{thm-LY-set}. In section 3, we prove Theorem \ref{thm-odd}, Theorem \ref{thm-even} and Theorem \ref{thm-harnack}. In Section 4, the Appendix, we give the calculations for comparison of Li-Xu's estimate \eqref{eq-LX} with Hamilton's estimate \eqref{eq-Ha} and for comparison of Li-Xu's estimate \eqref{eq-LX} with the Li-Yau-Davies estimate \eqref{eq-LYD}, and show \eqref{eq-dim-2}.
\section{Li-Yau multiplier set}
In this section, we give some simple properties of Li-Yau multiplier sets and prove Theorem \ref{thm-LY-set}.

First of all, we have the following elementary properties:
\begin{prop}\label{prop-ele-p}
Let $(M^n,g)$ be a complete Riemannian manifold and $u,v\in \mathcal P(M,g)$. Then, \begin{enumerate}
\item $S(u,x)$ is a closed subset of $(0,+\infty)\times[0,\infty)\times \R$.
\item $S_t(u,x)$ is convex.
\item if $(\beta,\gamma)\in S_t(u,x)$, then $[0,\beta]\times [\gamma,+\infty)\subset S_t(u,x)$.
\item $S(u+v,x)\supset S(u,x)\cap S(v,x)$.
\end{enumerate}
\end{prop}
\begin{proof}
The properties (1)--(3) are straight forward from definition. We only need to prove (4).

Let $(t,\beta,\gamma)\in S(u,x)\cap S(v,x)$. Then, at $(t,x)$,

\begin{equation*}
\begin{split}
&\beta\|\nabla \log(u+v)\|^2-(\log(u+v))_t\\
=&\frac{(\beta\|\nabla \log u\|^2-(\log u)_t)u^2+(\beta\|\nabla \log v\|^2-(\log v)_t)v^2+2\beta\vv<\nabla u,\nabla v>-u_tv-uv_t}{(u+v)^2}\\
\leq&\frac{\gamma u^2+\gamma v^2+(2\beta\vv<\nabla \log u,\nabla\log v>-(\log u)_t-(\log v)_t)uv}{(u+v)^2}\\
\leq&\frac{\gamma u^2+\gamma v^2+(\beta\|\nabla\log u\|^2-(\log u)_t+\beta\|\nabla \log v\|^2-(\log v)_t)uv}{(u+v)^2}\\
\leq& \gamma.
\end{split}
\end{equation*}
So, $(t,\beta,\gamma)\in S(u+v,x)$.
\end{proof}

We also have the following elementary properties for the Li-Yau multiplier set of a complete Riemannian manifold.
\begin{prop}\label{prop-ele-p-S}
Let $(M^n,g)$ be a complete Riemannian manifold. Then,
\begin{enumerate}
\item $S(M,g,x)$ and $S(M,g)$ are closed in $(0,\infty)\times [0,\infty)\times \R$.
\item $S(M,g,x)$ and $S(M,g)$ are convex.
\item If $(\beta,\gamma)\in S_t(M,g,x)$, then $[0,\beta]\times [\gamma,+\infty)\subset S_t(M,g,x)$. As a consequence, if $(\beta,\gamma)\in S_t(M,g)$, then $[0,\beta]\times [\gamma,+\infty)\subset S_t(M,g)$.
\item Let $\varphi$ be an isometric transformation of $(M,g)$, then $S(M,g,x)=S(M,g,\varphi(x))$. As a consequence, if $(M,g)$ is homogeneous, then $S(M,g)=S(M,g,x)$ for any $x\in M$.
\item Let $(N,h)$ be another complete Riemannian manifold, and $\varphi:M\to N$ be a local isometry. Then $S(M,g,x)\subset S(N,h,\varphi(x))$ for any $x\in M$. As a consequence, $S(M,g)\subset S(N,h)$.
\item $S_{t_1}(M,g,x)\subset S_{t_2}(M,g,x)$ when $0<t_1<t_2$, for any $x\in M$. As a consequence, $S_{t_1}(M,g)\subset S_{t_2}(M,g)$ when $0<t_1<t_2$.
\item $\bigcap_{\tau>t}S_\tau(M,g,x)=S_t(M,g,x)$ for any $x\in M$ and $t>0$. As a consequence, $S_t(M,g)=\cap_{\tau>t}S_\tau(M,g)$.
\item Let $\lambda$ be a positive constant. Then, $(t,\beta,\gamma)\in S(M,\lambda^2g,x)$ if and only if $(\lambda^{-2}t,\beta,\lambda^2\gamma)\in S(M,g,x)$. As a consequence, $(t,\beta,\gamma)\in S(M,\lambda^2g)$ if and only if $(\lambda^{-2}t,\beta,\lambda^2\gamma)\in S(M,g)$.
\item Let $(N,h)$ be another complete Riemannian manifold. Then, $S(M\times N, g\times h,(x,y))\subset S(M,g,x)$ for any $x\in M$ and $y\in N$. As a consequence, $S(M\times N, g\times h)\subset S(M,g)$.
\item Let $(t,\beta,\gamma)\in S(M,g,x)$ and further assume that the Ricci curvature of $(M,g)$ is bounded from below. Then, for any positive solution $u\in C^\infty([0,t]\times M)$ of the heat equation on $M$,
    \begin{equation}
      [\beta\|\nabla \log u\|^2-(\log u)_t](t,x)\leq \gamma.
    \end{equation}
\end{enumerate}
\end{prop}
\begin{proof}
The statements (1)--(3) are clearly true by definition. We only give the proofs of (4)--(10).
\begin{enumerate}
\item[(4)] For any $(t,\beta,\gamma)\in S(M,g,x)$ and $u\in \mathcal P(M,g)$, since $\varphi$ is an isometry, $\varphi^*u\in \mathcal P(M,g)$,
    \begin{equation*}
    \begin{split}
    \left[\beta\|\nabla \log u\|^2-(\log u)_t\right](t,\varphi(x))=\left[\beta\|\nabla \log \varphi^*u\|^2-(\log\varphi^* u)_t\right](t,x)\leq\gamma.
    \end{split}
    \end{equation*}
    This means that $S(M,g,x)\subset S(M,g, \varphi(x))$. By applying this to $\varphi^{-1}$, we obtain the conclusion.
\item[(5)]The proof is the same as that of (4). Because a local isometry of complete Riemannian manifolds must be surjective, we obtain the consequence that $S(M,g)\subset S(N,h)$.
\item[(6)] Let $(\beta,\gamma)\in S_{t_1}(M,g,x)$ and $u\in\mathcal P(M,g)$. Let $v(t,x)=u(t+t_2-t_1,x)$. It is clear that $v\in \mathcal P(M,g)$. So
    \begin{equation*}
    \begin{split}
    [\beta\|\nabla \log u\|^2-(\log u )_t](t_2,x)=[\beta\|\nabla \log v\|^2-(\log v)_t](t_1,x)\leq\gamma
    \end{split}
    \end{equation*}
    which means that $(\beta,\gamma)\in S_{t_2}(M,g,x)$. So, $S_{t_1}(M,g)\subset S_{t_2}(M,g)$.
\item[(7)]Let $(\beta,\gamma)\in \bigcap_{\tau>t}S_\tau(M,g,x)$. Then, for any $u\in \mathcal P(M,g)$,
    \begin{equation}
    [\beta\|\nabla\log u\|^2-(\log u)_t](\tau,x)\leq \gamma
    \end{equation}
    for any $\tau>t$. Setting $\tau\to t^+$, we have
    \begin{equation}
    [\beta\|\nabla\log u\|^2-(\log u)_t](t,x)\leq \gamma.
    \end{equation}
    So $(\beta,\gamma)\in S_t(M,g,x)$. Combining this and (6), we get the conclusion.
\item[(8)] Let $(t,\beta,\gamma)\in S(M,\lambda^2g,x)$. For any $u\in \mathcal P(M,g)$, let $v(t,x)=u(\lambda^{-2}t,x)$. Then, $v\in \mathcal P(M,\lambda^2g)$. So,
    \begin{equation}
    [\beta\|\nabla_{\lambda^2g}\log v\|^2-(\log v)_t](t,x)\leq \gamma.
    \end{equation}
    This implies that
    \begin{equation}
    [\beta\|\nabla \log u\|^2-(\log u)_t](\lambda^{-2}t,x)\leq\lambda^2\gamma.
    \end{equation}
    So $(\lambda^{-2}t,\beta,\lambda^2\gamma)\in S(M,g,x)$. The converse can be proved similarly.
\item[(9)] For any $(t,\beta,\gamma)\in S(M\times N,g\times h,(x,y))$ and $u\in \mathcal P(M,g,x)$. Let $v(t,x,y)=u(t,x)$. Then $v\in \mathcal P(M\times N,g\times h)$. So
    \begin{equation*}
    \begin{split}
    [\beta\|\nabla \log u\|^2-(\log u)_t](t,x)= [\beta\|\nabla \log v\|^2-(\log v)_t](t,x,y)    \leq \gamma,
    \end{split}
    \end{equation*}
    and hence $(t,\beta,\gamma)\in S(M,g)$.
\item[(10)] When $M$ is compact, the conclusion is clear. When $M$ is noncompact, let $\Omega_1\subset\subset\Omega_2\subset\subset\cdots\subset\subset\Omega_k\subset\subset\cdots$
    be a sequence of domains exhausts $M$ and $\eta_k$ be a smooth function on $M$ with $\supp \eta_k\subset \Omega_{k+1}$, $0\leq\eta_k\leq 1$ and $\eta_k|_{\Omega_k}\equiv 1$. Let $u_k(t,x)$ be a bounded solution of the heat equation on $M$ with $u_k(0,x)=\eta_k(x)u(0,x)$. Then $u_k\in \mathcal P(M,g)$ and $u_k\to u$ as $k\to\infty$ smoothly by the uniqueness of nonnegative solutions for heat equations when Ricci curvature of $(M,g)$ is bounded from below (see \cite{Do,LY}). So, for any $(t,\beta,\gamma)\in S(M,g,x)$,
    \begin{equation}
      [\beta\|\nabla\log u_k\|^2-(\log u_k)_t](t,x)\leq \gamma.
    \end{equation}
    By setting $k\to \infty$ in the last inequality, we get the conclusion.
\end{enumerate}

\end{proof}
Next, we come to prove Theorem \ref{thm-LY-set}.
\begin{proof}[Proof of Theorem \ref{thm-LY-set}] Since $H(\cdot,\cdot,y)\in \mathcal P(M,g)$ for any $y\in M$,
we only need to show that $S(H(\cdot,\cdot,y),x)\subset S(u,x)$ for any $u\in \mathcal P(M,g)$ that is smooth up to $t=0$ by using the trick in the proof of (6) in Proposition \ref{prop-ele-p-S}. Moreover, by using the trick in the proof of (10) in Proposition \ref{prop-ele-p-S}, we only need to show that $S(H(\cdot,\cdot,y),x)\subset S(u,x)$ for any $u\in \mathcal P(M,g)$ that is smooth up to $t=0$ and $u(0,x)=f(x)$ is of compact support. Then,$u(t,x)=\int_M H(t,x,y)f(y)dy$ by uniqueness of positive solutions of the heat equations (see \cite{Do,LY}). So, for any $(t,\beta,\gamma)\in\bigcap_{y\in M}S(H(\cdot,\cdot,y),x)$, we have that, at $(t,x)$,
\begin{equation*}
\begin{split}
&\beta\|\nabla \log u\|^2-(\log u)_t\\
=&\left(\int_MH(t,x,y)f(y)dy\right)^{-2}\Bigg(\beta\int_M \int_M\vv<\nabla_x H(t,x,y),\nabla_x H(t,x,z)>f(y)f(z)dydz-\\
&\int_M\int_MH_t(t,x,y)H(t,x,z)f(y)f(z)dydz\Bigg)\\
=&\frac{1}{2}\left(\int_MH(t,x,y)f(y)dy\right)^{-2}\Bigg(2\beta\int_M \int_M\vv<\nabla_x H(t,x,y),\nabla_x H(t,x,z)>f(y)f(z)dydz\\
&-\int_M\int_M(H_t(t,x,y)H(t,x,z)+H(t,x,y)H_t(t,x,z))f(y)f(z)dydz\Bigg)\\
=&\frac{1}{2}\left(\int_MH(t,x,y)f(y)dy\right)^{-2}\times\\
&\Bigg(2\beta\int_M \int_M\vv<\nabla_x \log H(t,x,y),\nabla_x \log H(t,x,z)>H(t,x,y)H(t,x,z)f(y)f(z)dydz\\
&-\int_M\int_M((\log H(t,x,y))_t+(\log H(t,x,z))_t)H(t,x,y)H(t,x,z)f(y)f(z)dydz\Bigg)\\
\leq&\frac{1}{2}\left(\int_MH(t,x,y)f(y)dy\right)^{-2}\times\\
&\Bigg(\beta\int_M \int_M(\|\nabla_x \log H(t,x,y)\|^2+\|\nabla_x \log H(t,x,z)\|^2)H(t,x,y)H(t,x,z)f(y)f(z)dydz\\
&-\int_M\int_M((\log H(t,x,y))_t+(\log H(t,x,z))_t)H(t,x,y)H(t,x,z)f(y)f(z)dydz\Bigg)\\
\leq&\gamma.
\end{split}
\end{equation*}
This completes the proof of the theorem.
\end{proof}
Similarly as in the fundamental work of Li-Yau \cite{LY}, one has the following relation of Li-Yau multiplier set and Harnack inequality for positive solution of heat equation by the same argument as in \cite{LY} (see also \cite{LX}).
\begin{prop}\label{prop-harnack-g}
Let $(M^n,g)$ be a complete Riemannian manifold and let $(t,\beta(t),\gamma(t))$ with $t\in (a,b]$ be a curve in $S(M,g)$ where $0\leq a<b$. Then, for any  $u\in \mathcal P(M,g)$, $a< t_1<t_2\leq b$ and $x_1,x_2\in M$,
\begin{equation}
u(x_1,t_1)\leq u(x_2,t_2)\exp\left(\frac{r^2(x_1,x_2)}{4(t_2-t_1)^2}\int_{t_1}^{t_2}\frac{1}{\beta(t)}dt+\int_{t_1}^{t_2}\gamma(t)dt\right).
\end{equation}
\end{prop}
\section{Optimal Li-Yau type gradient estimate on hyperbolic spaces }
The strategy for proving Theorem \ref{thm-odd} and Theorem \ref{thm-even} is similar with that of \cite{DM}. We first prove Theorem \ref{thm-odd} using the recurrence relation \eqref{eq-re-s2}. Then prove Theorem \ref{thm-even} by using the recurrence relation \eqref{eq-re-s1}.

The same as in \cite{DM}, we write the heat kernel $K_n$ of $\mathbb H^n$ as
\begin{equation}\label{eq-K}
K_n(t,r)=(4\pi t)^{-\frac n2}e^{-\frac{(n-1)^2}{4}t-\frac{r^2}{4t}}\alpha_n(t,r).
\end{equation}
Then, by \eqref{eq-re-s2}, $\alpha_n(t,r)$ satisfies the following recursive identity:
\begin{equation}\label{eq-re-alpha}
\alpha_n=\frac{r}{\sinh r}\alpha_{n-2}-\frac{2t}{\sinh r}\frac{\p\alpha_{n-2}}{\p r}.
\end{equation}
Moreover, $\alpha_1=1$ and $\alpha_3=\frac{r}{\sinh r}$. Let $f_1=\frac{r}{\sinh r}$ and $\sigma=\cosh r$. Then, by \eqref{eq-re-alpha},
\begin{equation}\label{eq-re-alpha-sig}
 \alpha_n=f_1\alpha_{n-2}-2t\frac{\p\alpha_{n-2}}{\p \sigma}.
 \end{equation}
 As mentioned in \cite{DM}, by induction, it is not hard to see that
 \begin{equation}\label{eq-alpha-odd}
 \alpha_{2m+1}=\sum_{i=0}^{m-1} t^iP_{m,i}(f_1,f_2,\cdots,f_m)
 \end{equation}
with $P_{m,0}(T_1,T_2,\cdots,T_m)=T_1^m$ and $P_{m,m-1}(T_1,T_2,\cdots,T_m)=2^{m-1}T_m$, where $$f_{m+1}=-\frac{df_{m}}{d\sigma}=-\frac{1}{\sinh r}\frac{d f_m}{dr}$$
 for $m=1,2,\cdots$. Here $P_{m,i}(T_1,T_2,\cdots,T_m)$'s are polynomials with nonnegative coefficients. As mentioned in \cite{DM}, by that $\alpha_{2m+1}$ is positive, $f_m$ is decreasing and positive. By making more detailed analysis on $\alpha_{2m+1}$ and $f_m$, we have the following results which will be used later.

\begin{prop}\label{prop-f}Let $q_m=\frac{f_{m+1}}{f_m}$ for $m=1,2,\cdots$. Then
\begin{enumerate}
\item $q_m(0)=\frac{m^2}{2m+1}$ for $m=1,2,\cdots$;
\item $\lim_{r\to\infty}q_m(r)\cosh r=m$ for $m=1,2,\cdots$;
\item $q_m(r)\cosh r\leq m$ and as a consequence $0\leq (-\log f_m)_r\leq m$ for $m=1,2,\cdots$;
\item $0\leq \cosh r(q_{m+1}(r)-q_m(r))\leq 1$ and as a consequence, $0\leq-(\log q_m)_r\leq 1$  for $m=1,2,\cdots$;
\item $P_{m,i}(T_1,T_2,\cdots,T_m)$ is a homogenous polynomial of degree $m-i$ for $m=1,2,\cdots$ and $i=0,1,\cdots,m-1$;
\item $P_{m,i}(T_1,T_2,\cdots,T_{m})$ is a weighted homogenous polynomial of degree $m$ with nonnegative coefficients for $m=1,2,\cdots$ and $i=0,1,2,\cdots,m-1$, when counting the degree of $T_j$ as $j$ for $j=1,2,\cdots,m$.
\end{enumerate}
\end{prop}
\begin{proof}
By taking derivative to $f_1\sinh r=r$ with respect to $r$, one has
\begin{equation}\label{eq-f-1}
\sigma f_1-(\sigma^2-1)f_2=1.
\end{equation}
Taking $m^{th}$ derivative to the last equality with respect to $\sigma$, we have
 \begin{equation}\label{eq-f-rec}
 m^2f_m-(2m+1)\sigma f_{m+1}+(\sigma^2-1)f_{m+2}=0
 \end{equation}
 for any $m=1,2,\cdots$.\\
(1) Let $\sigma=1$, i.e. $r=0$ in \eqref{eq-f-rec}. We have
 \begin{equation}
 m^2 f_m(0)=(2m+1)f_{m+1}(0).
 \end{equation}
 So $q_{m}(0)=\frac{m^2}{2m+1}$.\\
(2) By the expression of $f_1$, it is clear that
\begin{equation}
\lim_{r\to\infty}\frac{\cosh r}{r}f_1=1.
\end{equation}
Then, by \eqref{eq-f-1},
\begin{equation}
\lim_{r\to\infty}\frac{\cosh^2r}{r}f_2=1.
\end{equation}
By \eqref{eq-f-rec} and induction, $\lim_{r\to\infty}\frac{\cosh^mr f_m}{r}=a_m$ exists. Moreover, \begin{equation}\label{eq-a}
m^2a_m-(2m+1) a_{m+1}+a_{m+2}=0
\end{equation}
with $a_1=a_2=1$. By \eqref{eq-a},
\begin{equation}
m(ma_m-a_{m+1})=(m+1)a_{m+1}-a_{m+2}.
\end{equation}
So, we have $ma_m-a_{m+1}=0$ for any $m=1,2,\cdots$. This implies that
\begin{equation}
\lim_{r\to\infty}q_{m}(r)\cosh r=m.
\end{equation}
(3) By \eqref{eq-f-rec},
\begin{equation}
\begin{split}
m(mf_m-\sigma f_{m+1})-\sigma ((m+1)f_{m+1}-\sigma f_{m+2})=f_{m+2}>0.
\end{split}
\end{equation}
Moreover
\begin{equation}
(mf_m-\sigma f_{m+1})_\sigma=-((m+1)f_{m+1}-\sigma f_{m+2}).
\end{equation}
Substituting this into the last inequality, we have
\begin{equation}
m(mf_m-\sigma f_{m+1})+\sigma(mf_m-\sigma f_{m+1})_\sigma>0.
\end{equation}
This implies that
\begin{equation}
\left(\sigma^m(mf_m-\sigma f_{m+1})\right)_\sigma>0.
\end{equation}
Therefore, by (1),
\begin{equation}
\sigma^m(mf_m-\sigma f_{m+1})\geq (mf_m-\sigma f_{m+1})|_{\sigma=1}=f_m(0)(m-q_{m}(0))=\frac{m^2+m}{2m+1}f_m(0)>0.
\end{equation}
Furthermore,
\begin{equation}
(f_m)_r+mf_m=-\sinh r f_{m+1}+mf_m\geq -\cosh r f_{m+1}+mf_m=f_m(m-\cosh r q_m)\geq 0.
\end{equation}
So, $0\leq (-\log f_m)_r\leq m$ by that $f_m$ is decreasing.\\
(4) By \eqref{eq-f-rec},
\begin{equation}\label{eq-qm}
\begin{split}
m^2-(2m+1)\sigma q_m+(\sigma^2-1)q_mq_{m+1}=0
\end{split}
\end{equation}
and
\begin{equation}\label{eq-qm+1}
(m+1)^2-(2m+3)\sigma q_{m+1}+(\sigma^2-1)q_{m+1}q_{m+2}=0.
\end{equation}
Taking subtraction of the last two equalities, we have
\begin{equation}
\begin{split}
q_{m+1}(q_m-q_{m+2})=&\frac{2m+1-2\sigma q_m}{\sigma^2-1}+\frac{(2m+3)\sigma}{\sigma^2-1}(q_m-q_{m+1})\\
>& \frac{(2m+3)\sigma}{\sigma^2-1}(q_m-q_{m+1})\\
\end{split}
\end{equation}
where we have used (3) in the last inequality. Then,
\begin{equation}
\begin{split}
(q_m-q_{m+1})_\sigma=&q_m^2-q_{m+1}^2-q_{m+1}(q_m-q_{m+2})\\
<&\left(q_m+q_{m+1}-\frac{(2m+3)\sigma}{\sigma^2-1}\right)(q_m-q_{m+1}).
\end{split}
\end{equation}
So,
\begin{equation}
\left((\sigma^2-1)^\frac{2m+3}{2}e^{-Q_m-Q_{m+1}}(q_m-q_{m+1})\right)_\sigma<0
\end{equation}
where $Q_m=\int q_md\sigma$. This implies that $q_m-q_{m+1}\leq 0$.

Furthermore, taking substraction of \eqref{eq-qm+1} and \eqref{eq-qm}, one has
\begin{equation}
\begin{split}
&(2m+1)(1+\sigma(q_m-q_{m+1}))\\
=&\frac{\sigma^2-1}{\sigma}q_{m+1}\left(\frac{2\sigma^2}{\sigma^2-1}+\sigma (q_{m}-q_{m+2})\right)\\
\geq&\frac{\sigma^2-1}{\sigma}q_{m+1}[(1+\sigma(q_m-q_{m+1}))+(1+\sigma(q_{m+1}-q_{m+2}))]
\end{split}
\end{equation}
So
\begin{equation}
q_{m+1}[1+\sigma(q_{m+1}-q_{m+2})]\leq \left(\frac{(2m+1)\sigma}{\sigma^2-1}-q_{m+1}\right)(1+\sigma(q_m-q_{m+1})).
\end{equation}
Then,
\begin{equation}\label{eq-q-d}
\begin{split}
&[1+\sigma(q_m-q_{m+1})]_\sigma\\
=&(q_m-q_{m+1})+\sigma(q_m(q_m-q_{m+1})-q_{m+1}(q_{m+1}-q_{m+2}))\\
=&q_m(1+\sigma(q_m-q_{m+1}))-q_{m+1}(1+\sigma (q_{m+1}-q_{m+2}))\\
\geq&\left(q_m+q_{m+1}-\frac{(2m+1)\sigma}{\sigma^2-1}\right)(1+\sigma(q_m-q_{m+1})).
\end{split}
\end{equation}
Similarly as before, this implies that
\begin{equation}
1+\sigma(q_m-q_{m+1})\geq 0.
\end{equation}
Moreover, note that $(q_m)_\sigma=q_m(q_m-q_{m+1})\leq 0$. So, $q_m$ is decreasing. Furthermore,
\begin{equation}
\begin{split}
(q_m)_r+q_m=&\sinh r(q_m)_\sigma+q_m\\
=&[\sinh r(q_m-q_{m+1})+1]q_m\\
\geq&[\cosh r(q_m-q_{m+1})+1]q_m\\
\geq&0.
\end{split}
\end{equation}
So $0\leq-(\log q_m)_r\leq 1$.\\
(5) By \eqref{eq-re-alpha-sig} and \eqref{eq-alpha-odd},
\begin{equation}\label{eq-re-P}
P_{m+1,i}=T_1P_{m,i}+2\sum_{j=1}^m\frac{\p P_{m,i-1}}{\p T_j}T_{j+1}
\end{equation}
for $i=0,1,\cdots,m$. Here, we take $P_{m,-1}=P_{m,m}=0$. The conclusion follows by induction on $m$ and \eqref{eq-re-P}.\\
(6) The conclusion follows by induction on $m$ and \eqref{eq-re-P}.

\end{proof}

By the analysis of $f_m$ in Proposition \ref{prop-f}, we have the following estimates of $\alpha_n$.
\begin{prop}\label{prop-alpha}
\begin{enumerate}
\item For $m=1,2,\cdots$,
\begin{equation}
0\leq (\log\alpha_{2m+1})_t\leq \frac{m-1}{t}
\end{equation}
for all $t>0$. Hence
\begin{equation}
\lim_{t\to\infty}(\log\alpha_{2m+1})_t=0
\end{equation}
uniformly for $r\geq 0$.
\item For $m=1,2,\cdots$, $0\leq-(\log\alpha_{2m+1})_r\leq m$ and $$\lim_{r\to\infty}-(\log\alpha_{2m+1})_r=m$$
uniformly for $t>0$.

\item For $m=1,2,\cdots$,
\begin{equation}0\leq -\left(\log \alpha_{2m}\right)_r\leq m-\frac{1}{2}.\end{equation}

\end{enumerate}
\end{prop}
\begin{proof}
(1) By the expression \eqref{eq-alpha-odd},
    \begin{equation}
    0\leq(\log \alpha_{2m+1})_t=\frac{\sum_{i=1}^{m-1}it^{i-1}P_{m,i}(f_1,f_2,\cdots,f_m)}{\sum_{i=0}^{m-1}t^iP_{m,i}(f_1,f_2,\cdots,f_m)}\leq \frac{m-1}{t}.
    \end{equation}
\noindent (2) Since $f_m$ is decreasing for $m=1,2,\cdots$, by \eqref{eq-alpha-odd}, $\alpha_{2m+1}$ is decreasing with respect to $r$. So $-(\log\alpha_{2m+1})_r\geq 0$. On the other hand, by (6) of Proposition \ref{prop-f}, suppose that
    \begin{equation}
    P_{m,i}(f_1,f_2,\cdots,f_m)=\sum_{\tiny\begin{array}{l}j_1,j_2,\cdots,j_m\geq0\\j_1+2j_2+\cdots+mj_m=m\end{array}}a_{m,j_1,j_2,\cdots,j_m}f_1^{j_1}f_2^{j_2}\cdots f_m^{j_m}
    \end{equation}
    with $a_{m,j_1,j_2,\cdots,j_m}\geq 0$. Then
    \begin{equation}
    \begin{split}
    &(\alpha_{2m+1})_r+m\alpha_{2m+1}\\
    =&\sum_{i=0}^{m-1} t^i[(P_{m,i}(f_1,f_2,\cdots,f_m))_r+mP_{m,i}(f_1,f_2,\cdots,f_m)]\\
    =&\sum_{i=0}^{m-1}t^i\sum_{\tiny\begin{array}{l}j_1,j_2,\cdots,j_m\geq0\\j_1+2j_2+\cdots+mj_m=m\end{array}}a_{m,j_1,j_2,\cdots,j_m}[(f_1^{j_1}f_2^{j_2}\cdots f_m^{j_m})_r+mf_1^{j_1}f_2^{j_2}\cdots f_m^{j_m}]\\
    =&\sum_{i=0}^{m-1}t^i\sum_{\tiny\begin{array}{l}j_1,j_2,\cdots,j_m\geq0\\j_1+2j_2+\cdots+mj_m=m\end{array}}a_{m,j_1,j_2,\cdots,j_m}f_1^{j_1}f_2^{j_2}\cdots f_m^{j_m}\sum_{k=1}^mj_k\left(k+(\log f_k)_r\right)\\
    \geq&0
    \end{split}
    \end{equation}
    by (3) of Proposition \ref{prop-f}. So, $-(\log\alpha_{2m+1})_r\leq m$. Moreover, note that
    \begin{equation}
    \begin{split}
   -\frac{(f_1^{j_1}f_2^{j_2}\cdots f_m^{j_m})_r}{f_1^{j_1}f_2^{j_2}\cdots f_{m}^{j_m}}=&j_1q_1\sinh r+j_2q_2\sinh r+\cdots+ j_mq_m\sinh r\\
   \to& j_1+2j_2+\cdots +mj_m
   \end{split}
    \end{equation}
    as $r\to \infty$, by (2) of Proposition \ref{prop-f}. From this, \eqref{eq-alpha-odd} and (6) of Proposition \ref{prop-f}, we get the conclusion.\\
(3) By \eqref{eq-re-s1},
\begin{equation}\label{eq-alpha-even}
\begin{split}
\alpha_{2m}(t,r)=&\frac{\sqrt 2 }{(4\pi t)^\frac12}\int_{r}^\infty\frac{\alpha_{2m+1}(t,s)e^{-\frac{s^2-r^2}{4t}}\sinh s }{\sqrt{\cosh s-\cosh r}}ds\\
=&\frac{1}{(8\pi t)^\frac12}\int_{0}^\infty\frac{\alpha_{2m+1}(t,s)e^{-\frac{x}{4t}}/f_1(s) }{\sqrt{\cosh s-\cosh r}}dx\\
\end{split}
\end{equation}
where $s=\sqrt{x+r^2}$.  By that $q_m$ is decreasing and the expression \eqref{eq-alpha-odd}, $\alpha_{2m+1}(t,s)/f_1(s)$ is decreasing on $r$. Combining this with Lemma \ref{lem-Z}, we know that $\alpha_{2m}(t,r)$ is decreasing with respect to $r$.

Furthermore, by (4) of Proposition \ref{prop-f},
\begin{equation}
0\leq -\left(\log \frac{f_{m}}{f_1}\right)_r\leq m-1
\end{equation}
for $m=1,2,\cdots$. From this and using the same argument as in the proof (2), we have
\begin{equation}
-[\log(\alpha_{2m+1}(t,s)/f_1(s))]_s\leq m-1.
\end{equation}
Therefore,
\begin{equation}
-[\log(\alpha_{2m+1}(t,s)/f_1(s))]_r=-[\log(\alpha_{2m+1}(t,s)/f_1(s))]_s\frac{r}{s}\leq m-1.
\end{equation}
Then, by Lemma \ref{lem-Z}, we know that
\begin{equation}
 -[\log(\alpha_{2m+1}(t,s)(\cosh s-\cosh r)^{-\frac{1}2}/f_1(s))]_r \leq m-\frac12.
\end{equation}
From this and  the expression \eqref{eq-alpha-even} of $\alpha_{2m}$, we get the conclusion.\\

\end{proof}
In the proof of Proposition \ref{prop-alpha}, we need the following lemma.
\begin{lem}\label{lem-Z} For any positive constant $a$
$$0\leq\left[\log\left(\cosh\sqrt{a^2+r^2}-\cosh r\right)\right]_r\leq 1.$$
\end{lem}
\begin{proof}
  Note that
  \begin{equation}
  \left(\cosh\sqrt{a^2+r^2}-\cosh r\right)_r=r\left(\frac{\sinh \sqrt{a^2+r^2}}{\sqrt{a^2+r^2}}-\frac{\sinh r}{r}\right)\geq 0
  \end{equation}
  by that $\frac{\sinh x}{x}$ is an increasing function. So,
  \begin{equation}
  \left[\log\left(\cosh\sqrt{a^2+r^2}-\cosh r\right)\right]_r\geq 0.
  \end{equation}
Moreover, to show that
\begin{equation}
\left[\log\left(\cosh\sqrt{a^2+r^2}-\cosh r\right)\right]_r\leq 1
\end{equation}
it is equivalent to show that
\begin{equation}
F(\rho,r):=\rho(\cosh \rho-\cosh r)-(r\sinh \rho-\rho\sinh r)\geq 0
\end{equation}
where $\rho=\sqrt{a^2+r^2}\geq r$. Note that
\begin{equation}
\p_\rho F=\rho\sinh \rho+\cosh \rho-r\cosh \rho -\cosh r+\sinh r
\end{equation}
and
\begin{equation}
\p_\rho^2F=2\sinh\rho+\rho\cosh\rho-r\sinh\rho>0
\end{equation}
for $\rho\geq r$. Hence
\begin{equation}
\p_\rho F(\rho,r)\geq\p_\rho F(r,r)=(1+r)\sinh r-r\cosh r.
\end{equation}
Moreover
\begin{equation}
[\p_\rho F(r,r)]_r=r\cosh r+\sinh r-r\sinh r\geq0.
\end{equation}
So,
\begin{equation}
\p_\rho F(\rho,r)\geq\p_\rho F(r,r)\geq\p_\rho F(0,0)=0.
\end{equation}
Hence
\begin{equation}
F(\rho,r)\geq F(r,r)=0.
\end{equation}
This completes the proof of the lemma.
\end{proof}
We are now ready to prove our main results.
\begin{proof}[Proof of Theorem \ref{thm-odd}]By (5) of Proposition \ref{prop-ele-p-S}, we only need to prove the theorem for hyperbolic spaces. By \eqref{eq-K}
\begin{equation}
\log K_n=-\frac n2\log(4\pi t)-\frac{(n-1)^2}{4}t-\frac{r^2}{4t}+\log\alpha_n.
\end{equation}
Then, by (1) and (2) of Proposition \ref{prop-alpha},
\begin{equation}\label{eq-LY-odd}
\begin{split}
&\beta\|\nabla \log K_n\|^2-(\log K_n)_t\\
=&\beta\left(-\frac{r}{2t}+(\log\alpha_n)_r\right)^2+\frac{n}{2t}+\frac{(n-1)^2}{4}-\frac{r^2}{4t^2}-(\log\alpha_n)_t\\
=&-(1-\beta)\left(\frac{r}{2t}+\frac{\beta}{1-\beta}(\log \alpha_n)_r\right)^2+\frac{\beta}{1-\beta}(-\log \alpha_n)_r^2+\frac{(n-1)^2}{4}+\frac{n}{2t}-(\log\alpha_n)_t\\
\leq&\frac{n}{2t}+\frac{(n-1)^2}{4(1-\beta)}.
\end{split}
\end{equation}

\end{proof}
\begin{rem}
By the asymptotic behaviors in (1) and (2) of Proposition \ref{prop-alpha}, we know that \eqref{eq-LY-odd} is asymptotically sharp as $t$ and $r$ tending to infinity for hyperbolic spaces.
\end{rem}
\begin{proof}[Proof of Theorem \ref{thm-odd}]By (5) of Proposition \ref{prop-ele-p-S}, we only need to prove the theorem for hyperbolic spaces. By \eqref{eq-K}
\begin{equation}
\log K_n=-\frac {n+1}2\log(4\pi t)-\frac{(n-1)^2}{4}t-\frac{r^2}{4t}+\log((4\pi t)^\frac12\alpha_n).
\end{equation}
By \eqref{eq-alpha-even} and (1) of Proposition \ref{prop-alpha},
\begin{equation}
\left[\log((4\pi t)^\frac12\alpha_n)\right]_t\geq 0.
\end{equation}
Then, by (3) of Proposition \ref{prop-alpha},
\begin{equation}\label{eq-LY-even}
\begin{split}
&\beta\|\nabla \log K_n\|^2-(\log K_n)_t\\
=&\beta\left(-\frac{r}{2t}+(\log((4\pi t)^\frac12\alpha_n))_r\right)^2+\frac{n+1}{2t}+\frac{(n-1)^2}{4}-\frac{r^2}{4t^2}-(\log(4\pi t)^\frac12\alpha_n)_t\\
\leq &-(1-\beta)\left(\frac{r}{2t}+\frac{\beta}{1-\beta}(\log\alpha_n)_r\right)^2+\frac{\beta}{1-\beta}(-\log \alpha_n)_r^2+\frac{(n-1)^2}{4}+\frac{n+1}{2t}\\
\leq&\frac{n+1}{2t}+\frac{(n-1)^2}{4(1-\beta)}.
\end{split}
\end{equation}
\end{proof}
\begin{rem}
The difference of the odd dimensional case and even dimensional case in our argument is that $\alpha_n$ is not increasing with respect to $t$ when $n$ is even. In fact, one can see that $\alpha_2$ is decreasing with respect to $t$ (see Proposition \ref{prop-dim-2} in the Appendix.).
\end{rem}
By applying Proposition \ref{prop-harnack-g}, we are able to prove Theorem \ref{thm-harnack}.

\begin{proof}[Proof of Theorem \ref{thm-harnack}]
We only need to prove the odd dimensional case. The proof of the even dimensional case is similar.

By Proposition \ref{prop-harnack-g} and Theorem \ref{thm-odd}, we have
\begin{equation}
u(x_1,t_1)\leq \left(\frac{t_2}{t_1}\right)^\frac {n}2 \exp\left(\frac{r^2(x_1,x_2)}{4\beta(t_2-t_1)}+\frac{(n-1)^2(t_2-t_1)}{4(1-\beta)}\right)u(x_2,t_2)
\end{equation}
for any constant $\beta\in (0,1)$. Let $\beta=\frac{1}{1+\frac{(n-1)(t_2-t_1)}{r(x_1,x_2)}}$ which is the minimum point of $\frac{r^2(x_1,x_2)}{4\beta(t_2-t_1)}+\frac{(n-1)^2(t_2-t_1)}{4(1-\beta)}$. We get the conclusion.
\end{proof}

\section{Appendix}
In this appendix, we give the details in the  comparisons of  Li-Xu's estimate \eqref{eq-LX} with the Li-Yau-Davies estimate \eqref{eq-LYD} and with Hamilton's estimate \eqref{eq-Ha} respectively which are not that obvious comparing to the other comparisons of estimates in Section 1. Moreover, we will show that $\alpha_2$ is deceasing with respect to $t$.

We first compare Li-Xu's estimate \eqref{eq-LX} and Hamilton's estimate \eqref{eq-Ha}.
\begin{prop}\label{prop-LXHa}
For any $x\geq 0$, $$\frac{ x(\coth x+1)}{1+\frac{\sinh x\cosh x-x}{\sinh^2x}}\leq e^{2x}$$ and $$1+\frac{\sinh x\cosh x-x}{\sinh^2x}\leq e^{2x}.$$
As a consequence, Li-Xu's estimate \eqref{eq-LX} is better than Hamilton's estimate \eqref{eq-Ha} by letting $x=kt$.
\end{prop}
\begin{proof}
Let $\beta=\frac{1}{1+\frac{\sinh x\cosh x-x}{\sinh^2x}}$. Then, it is clear that \begin{equation}
\begin{split}
\frac{ x(\coth x+1)}{1+\frac{\sinh x\cosh x-x}{\sinh^2x}}=&x\left(\frac{1}{\beta}+\frac{x}{\sinh^2 x}\right)\beta\\
=&x+\frac{x^2}{\sinh ^2x}\beta\\
\leq&1+x\\
\leq&e^{2x}.
\end{split}
\end{equation}
For the other inequality, it is equivalent to
\begin{equation}
f(x)=(e^{2x}-1)\sinh^2 x-\cosh x\sinh x+x\geq 0.
\end{equation}
Note that
\begin{equation}
f'(x)=2(e^{2x}-1)\sinh^2x+2(e^{2x}-1)\sinh x\cosh x\geq 0.
\end{equation}
So
\begin{equation}
f(x)\geq f(0)=0.
\end{equation}
\end{proof}
Next, we come to the comparison of Li-Xu's estimate \eqref{eq-LX} and the Li-Yau-Davies estimate \eqref{eq-LYD}.
\begin{prop}\label{prop-LXLY}Let $\beta=\frac{1}{1+\frac{\sinh x\cosh x-x}{\sinh^2x}}$. Then,
\begin{enumerate}
\item $x(\coth x+1)\beta< \frac{1}{\beta}+\frac{x}{2(1-\beta)}$ for any $x>0$. As a consequence, $\gamma_{LYD}(\beta_{LX}(t), t)>\gamma_{LX}(t)$ for any $t>0$ by setting $x=kt$.
\item The graphs of the functions $(1+\sqrt{\frac{x}2})^2$ and $x(\coth(x)+1)\beta$ intersect at only one point $x_{LX}\geq 8$. As a consequence, the graphs of the functions $\gamma_m(t)$ and $\gamma_{LX}(t)$ intersect at only one point $t_{LX}=\frac{x_{LX}}{k}$.
\item When $x\geq x_{LX}$, $\beta>\frac{1}{1+\sqrt{\frac x2}}.$ As a consequence, $\beta_m(t)<\beta_{LX}(t) $ for $t\geq t_{LX}$.
\end{enumerate}
\end{prop}
\begin{proof}
\begin{enumerate}
\item It is not hard to see that $\frac 12\leq\beta\leq 1$. Then, the same as in the proof of the last proposition, we have
    \begin{equation}
    x(\coth x+1)\beta\leq 1+x< \frac{1}{\beta}+\frac{x}{2(1-\beta)}.
    \end{equation}
\item The same as in the proof of the last proposition, we have
\begin{equation}
x\leq x(\coth x+1)\beta\leq x+1.
\end{equation}
The graphs of $(1+\sqrt{\frac x 2})^2$ and $x+1$ intersect at $x=8$ while the graphs of $(1+\sqrt{\frac x 2})^2$  and $x$ intersect at $x=6+4\sqrt 2$. Hence, the graphs of $(1+\sqrt{\frac x 2})^2$ and $x(\coth x+1)\beta$ must intersect at some  $x\in [8,6+4\sqrt 2]$. Moreover, note that
\begin{equation}
\left[\left(1+\sqrt{\frac x2}\right)^2\right]_x\leq \frac{3}{4}
\end{equation}
when $x\geq 8$ while
\begin{equation}
\begin{split}
&[x(\coth x+1)\beta]_x\\
=&\left[x+\frac{x^2}{\sinh^2 x}\beta\right]_{x}\\
=&1+\left(\frac{2x}{\sinh^2x}-\frac{2x^2\cosh x}{\sinh^3x}\right)\beta+\frac{x^2}{\sinh^2x}\left(\frac{2}{\sinh^2x}-\frac{2x\cosh x}{\sinh^3x}\right)\beta^2\\
\geq&1-\frac{(2x^2+2x^3)\cosh x}{\sinh^3x}\\
\geq&1-\frac{(4x^2+4x^3)}{\sinh^2x}\\
\geq&1-\frac{4x+4x^2}{\sinh x}\\
>&\frac 34
\end{split}
\end{equation}
when $x\geq 8$. Here we have used that $0<\beta\leq 1$, $\frac{x}{\sinh x}\leq 1$ and $\cosh x\leq 2\sinh x$ for $x\geq 8$. This gives us the conclusion.
\item When $x\geq x_{LX}\geq 8$,
\begin{equation}
\sqrt\frac x2\sinh^2x-\cosh x\sinh x+x\geq 2\sinh^2x-\cosh x\sinh x+x>0
\end{equation}
since $2\sinh x\geq \cosh x$ when $x\geq 8$. This gives us the conclusion.
\end{enumerate}
\end{proof}
Finally, we come to show that $\alpha_2$ is decreasing with respect to $t$.
\begin{prop}\label{prop-dim-2}
For the hyperbolic plane $\mathbb H^2$, $(\log \alpha_2)_t<0$ and hence
\begin{equation}
-(\log K_2)_t(t,0)>\frac{1}{t}+\frac{1}{4}.
\end{equation}
\end{prop}
\begin{proof}By \eqref{eq-alpha-even},
\begin{equation}
\begin{split}
\alpha_2=&\frac{\sqrt 2e^{\frac{r^2}{4t}}}{\sqrt{4\pi t}}\int_{r}^\infty\frac{se^{-\frac{s^2}{4t}}}{\sqrt{\cosh s-\cosh r}}ds\\
=&\frac{1}{2\sqrt{2\pi }}\int_{0}^\infty\frac{e^{-\frac x4}}{\sqrt{\frac{\cosh (\sqrt{tx+r^2})-\cosh r}{t}}}dx\\
\end{split}
\end{equation}
where $x=\frac{s^2-r^2}{t}$. Note that
\begin{equation}
\begin{split}
&\frac{\cosh (\sqrt{tx+r^2})-\cosh r}{t}\\
=&\frac{1}{t}\sum_{i=1}^\infty\frac{(tx+r^2)^i-r^{2i}}{(2i)!}\\
=&\frac{x}{2}+a_1(x,r) t+a_2(x,r) t^2+\cdots
\end{split}
\end{equation}
with $a_1(x,r),a_2(x,r),\cdots$  all positive numbers when $x>0$, hence it is strictly increasing. So, $(\alpha_2)_t<0$. Moreover, by \eqref{eq-K}, we have
\begin{equation}
-(\log K_2)_t(t,0)=\frac{1}{t}+\frac{1}{4}-(\log\alpha_2)_t(t,0)>\frac{1}{t}+\frac{1}{4}.
\end{equation}
\end{proof}

\end{document}